\documentclass[12pt]{article}

\usepackage{amssymb}
\usepackage{amsmath}
\usepackage{cite}
\usepackage{microtype}
\usepackage{mathrsfs}

\begin{document}

\renewcommand{\citeleft}{{\rm [}}
\renewcommand{\citeright}{{\rm ]}}
\renewcommand{\citepunct}{{\rm,\ }}
\renewcommand{\citemid}{{\rm,\ }}

\newcommand{\LegendreScalarProduct}[3][n]{\left[ #2, #3 \right]_{#1}}

\newcounter{abschnitt}
\newtheorem{satz}{Theorem}
\newtheorem{theorem}{Theorem}[abschnitt]
\newtheorem{koro}[theorem]{Corollary}
\newtheorem{prop}[theorem]{Proposition}
\newtheorem{lem}[theorem]{Lemma}
\newtheorem{approp}{Proposition}[section]
\newtheorem{apthm}[approp]{Theorem}

\newcounter{saveeqn}
\newcommand{\alpheqn}{\setcounter{saveeqn}{\value{abschnitt}}
\renewcommand{\theequation}{\mbox{\arabic{saveeqn}.\arabic{equation}}}}
\newcommand{\reseteqn}{\setcounter{equation}{0}
\renewcommand{\theequation}{\arabic{equation}}}

\hyphenpenalty=7500

\sloppy

\phantom{a}

\vspace{-2.2cm}

\begin{center}
\begin{Large} {\bf Radon Transforms of Projection Functions and Aleksandrov--Fenchel Inequalities} \\[0.7cm] \end{Large}

\begin{large} Leo Brauner, Oscar Ortega-Moreno, Franz E.\ Schuster \end{large}
\end{center}

\vspace{0.1cm}

\emph{Dedicated to Prof.\ Erwin Lutwak on the occasion of his 80th birthday}

\vspace{-0.8cm}

\begin{quote}
\footnotesize{ \vskip 1truecm\noindent {\bf Abstract.} The Radon transforms on Grassmannians define two families of geometric operators on convex bodies when applied to their projection functions. On the one hand, these are Minkowski's projection body maps and on the other hand these operators coincide on origin-symmetric bodies with the mean section body maps introduced by Goodey and Weil in 1992. Over the past two decades various sharp isoperimetric inequalities for projection bodies were shown to also hold for mean section bodies. However, an analogue of Lutwak's log-concavity inequalities for the intrinsic volumes of projection bodies remained open for mean section bodies. We establish new Aleksandrov--Fenchel inequalities for averages of certain mixed volumes and apply them to show that also all the intrinsic volumes of mean section bodies are indeed log-concave.}
\end{quote}

\vspace{0.6cm}

\centerline{\large{\bf{ \setcounter{abschnitt}{1}
\arabic{abschnitt}. Introduction}}} \alpheqn

\vspace{0.6cm}

Radon transforms play a fundamental role in several areas of mathematics such as harmonic and functional analysis, integral geometry, and geometric tomography (see, e.g., the books \textbf{\cite{Gardner2006, Helgason99, Koldobsky05, KoldobskyYaskin}}). In this article we are concerned with Radon transforms of functions on the Grassmannians $\mathrm{Gr}_{i,n}$ of $i$-dimensional subspaces of Euclidean space $\mathbb{R}^n$, where throughout $n \geq 3$. For $1 \leq i, j \leq n - 1$ and $F \in \mathrm{Gr}_{j,n}$, let $\mathrm{Gr}_{i,n}^F$ denote the submanifold of
$\mathrm{Gr}_{i,n}$ which comprises of all $E \in \mathrm{Gr}_{i,n}$ that contain (are contained in, respectively) $F$. The
{\it Radon transform} $R_{i,j}: L^2(\mathrm{Gr}_{i,n}) \rightarrow L^2(\mathrm{Gr}_{j,n})$ is defined by
\[(R_{i,j}f)(F) = \int_{\mathrm{Gr}_{i,n}^F}f(E)\,d\nu_i^F(E),   \]
where $\nu_i^F$ is the unique invariant probability measure on $\mathrm{Gr}_{i,n}^F$.

Let $\mathcal{K}^n$ denote the space of convex bodies (that is, compact, convex subsets) in $\mathbb{R}^n$ endowed with the Hausdorff metric and recall that $K \in \mathcal{K}^n$ is determined by its support function $h(K,u)=\max\{u\cdot x: x \in K\}$, for $u \in S^{n-1}$. The $i$th projection function $\mathrm{vol}_i(K|\,\cdot\,)$ of $K \in \mathcal{K}^n$ is the continuous function on $\mathrm{Gr}_{i,n}$ defined such that $\mathrm{vol}_i(K|E)$, for $E \in \mathrm{Gr}_{i,n}$, is the $i$-dimensional volume of the orthogonal projection of $K$ onto $E$. For $1 \leq i \leq n - 1$, applying the Radon transforms $R_{i,1}$ to $\mathrm{vol}_i(K|\,\cdot\,)$, defines a family of maps $\overline{\mathrm{M}}_i : \mathcal{K}^n \rightarrow \mathcal{K}^n$ by
\begin{equation} \label{defMimaps}
h(\overline{\mathrm{M}}_i K,\,\cdot\,) = R_{i,1} \mathrm{vol}_i(K|\,\cdot\,),
\end{equation}
where we identify an even function on $S^{n-1}$ with a function on $\mathrm{Gr}_{1,n}$. It was first shown by Goodey and Weil \textbf{\cite{GoodeyWeil1992}} that the right hand side of (\ref{defMimaps}) is indeed the support function of a convex body. In fact Goodey and Weil considered different operators they called mean section body maps and, up to constants, the family $\overline{\mathrm{M}}_i$ turned out to be their even part (see \textbf{\cite{GoodeyWeil1992, Goodey1998, GoodeyWeil2012, GoodeyWeil2014}} and Section 2 for details). In this article we will therefore frequently refer to the convex bodies $\overline{\mathrm{M}}_i K$ as mean section bodies as well.

A different set of operators between convex bodies which arises from an application of the Radon transforms to projection functions are the projection body maps $\overline{\Pi}_i : \mathcal{K}^n \rightarrow \mathcal{K}^n$, given for $1 \leq i \leq n - 1$ by
\begin{equation} \label{defPimaps}
h(\overline{\Pi}_i K,\,\cdot\,) = R_{n-i,1} \mathrm{vol}_i^{\bot}(K|\,\cdot\,).
\end{equation}
Here, $\mathrm{vol}_i^{\bot}(K|\,\cdot\,)$ denotes the continuous
function on $\mathrm{Gr}_{n-i,n}$ given by
$\mathrm{vol}_i^{\bot}(K|E)=\mathrm{vol}_i(K|E^{\bot})$.
Up to normalization, these operators are the projection body maps
introduced by Minkowski. Projection bodies have since become central
objects in convex geometry, functional analysis, and geometric tomography
(see, e.g., \textbf{\cite{Gardner2006, schneider93}}).

A major reason for the prominent role of projection bodies is the sharp
affine isoperimetric inequality established by Petty
\textbf{\cite{petty67}} and its extensions to projection bodies of all
orders due to Lutwak \textbf{\cite{Lutwak1985b}}. To recall these
inequalities, for a convex body $L$ containing the origin in its interior,
let $L^\circ=\{x\in\mathbb{R}^n:x\cdot y\leq1\text{ for all }y\in L\}$
denote its \emph{polar body}. If $K\in\mathcal{K}^n$ has non-empty
interior and $1\leq i\leq n-1$, then
\begin{equation} \label{WVnPiiVnkette}
\frac{\kappa_n^{n+1}}{\kappa_{n-1}^n}V_n(\Pi_i^\circ K)^{-1}
\geq\kappa_n^{n-i}V_n(K)^i,
\end{equation}
with equality if and only if $K$ is an ellipsoid when $i=n-1$, and a
Euclidean ball when $i\leq n-2$. Here, $\Pi_i$ is the projection body
map of order $i$, normalized such that $\Pi_iB=\kappa_{n-1}B$, and
$\Pi_i^\circ K=(\Pi_iK)^\circ$. Thus $\Pi_i$ differs from
$\overline{\Pi}_i$ only by a positive constant. The Lutwak--Petty
inequalities strengthen the classical inequalities relating the
quermassintegrals $W_{n-i}(K)$ to the volume of $K$ (see Section 2).
Related inequalities for mixed bodies and projection measures were
obtained by Lutwak \textbf{\cite{Lutwak1986MixedBodies}}. For connections
between the Petty projection, Blaschke--Santal\'o, and other affine
isoperimetric inequalities, see
\textbf{\cite{Lutwak1986Affine, Lutwak1993Selected}}.

Analogues of (\ref{WVnPiiVnkette}) for mean section bodies were obtained
more recently. With the normalization in (\ref{defMimaps}), for
$K\in\mathcal{K}^n$ with non-empty interior and $1\leq i\leq n-1$,
these inequalities take the form
\begin{equation} \label{meansectionpolariso}
W_{n-i}(K)^n
\geq\frac{\kappa_n^{n+1}}{\kappa_i^n}
V_n((\overline{\mathrm{M}}_iK)^\circ)^{-1}
\geq\kappa_n^{n-i}V_n(K)^i.
\end{equation}
For $2\leq i\leq n-1$, equality on the right holds if and only if
$K$ is a Euclidean ball; for $i=1$, it holds precisely for ellipsoids.
The case $i=n-1$ follows from results in
\textbf{\cite{HaberlSchu2019}}, while the intermediate degrees were
settled in \textbf{\cite{HofstaetterKniefaczSchuster2025}}. The case $i=1$ is a
consequence of the Blaschke--Santal\'o inequality.

The corresponding problem of finding sharp lower bounds for intrinsic
volumes of projection bodies remains largely open. Lutwak
\textbf{\cite{Lutwak1990b}} conjectured that the ratios
$V_{i+1}(\Pi_iK)/V_{i+1}(K)^i$, $1\leq i\leq n-1$, are minimized
by Euclidean balls and showed that this would follow from Petty's
conjectured projection inequality, studied further by Lutwak
\textbf{\cite{Lutwak1990Petty}}. When $i=n-1$, ellipsoids are also
minimizers by affine invariance. Lutwak confirmed the conjecture for
$i=1$. Only recently, Petty's conjecture has been proved in some special cases, see \textbf{\cite{Mielke2026,ChenFengLiXiXu2026}}.
For another recent approach to these isoperimetric problems
through fixed points of Minkowski valuations, see
\textbf{\cite{OrtegaMorenoSchuster2021,OrtegaMoreno2021,Ivaki2017,Ivaki2018,SaroglouZvavitch2017}}. These questions also motivated the study of Brunn--Minkowski
inequalities for projection bodies. In particular, Lutwak
\textbf{\cite[\textnormal{Theorem 6.2}]{lutwak93}} proved that, for
$K,L\in\mathcal{K}^n$ with non-empty interiors, $1\leq i\leq n$,
$2\leq j\leq n-1$, and $\lambda\in(0,1)$,
\begin{equation} \label{Piilogconc}
V_i(\overline{\Pi}_j((1-\lambda)K+\lambda L))
\geq V_i(\overline{\Pi}_jK)^{1-\lambda}
V_i(\overline{\Pi}_jL)^\lambda,
\end{equation}
with equality if and only if $K$ and $L$ are translates of each other.

While (\ref{Piilogconc}) holds for all intrinsic volumes of projection
bodies of arbitrary order, its analogues for mean section bodies have
previously been known only in restricted ranges. Our first main result
establishes the complete family of these inequalities.

\begin{satz} \label{thm1}
Let $1 \leq i \leq n$ and $2 \leq j \leq n - 1$. If
$K,L \in \mathcal{K}^n$ have non-empty interiors, then for every
$\lambda \in (0,1)$,
\begin{equation} \label{thm4form}
V_i(\overline{\mathrm{M}}_j(\lambda K + (1-\lambda)L))
\geq V_i(\overline{\mathrm{M}}_jK)^\lambda
V_i(\overline{\mathrm{M}}_jL)^{1-\lambda},
\end{equation}
with equality if and only if $K$ and $L$ are translates of each other.
\end{satz}

The first intrinsic volume provides a direct connection between
(\ref{Piilogconc}), Theorem \ref{thm1}, and the classical
Brunn--Minkowski theory. Indeed, for both families of operators,
$V_1$ of the image is a positive constant multiple of $V_j$ of the
original body. Thus the inequalities for $i=1$ follow from the
multiplicative Brunn--Minkowski inequality for intrinsic volumes.
For the volume functional itself, this fundamental log-concavity
property reads
\begin{equation} \label{logbmineq}
V_n((1-\lambda)K+\lambda L)
\geq V_n(K)^{1-\lambda}V_n(L)^\lambda,
\end{equation}
where $K,L\in\mathcal{K}^n$ have non-empty interiors and
$\lambda\in(0,1)$. Equality holds if and only if $K$ and $L$ are
translates. For an account of the many consequences of the
Brunn--Minkowski inequality, including the classical Euclidean
isoperimetric inequality, we refer to Gardner
\textbf{\cite{gardner02}}.

Several earlier results on Minkowski valuations imply special cases
of Theorem \ref{thm1}. In particular, the case $j=n-1$ follows from
\textbf{\cite{Schu06b}}, and the range $1\leq i\leq j+1$ is covered
by inequalities from \textbf{\cite{BPSW2014}}, extending earlier work
in \textbf{\cite{SchusterCrofton2010, ParapatitsSchuster2012}}.
Theorem \ref{thm1} removes the restriction on the intrinsic volume and includes a complete characterization of
equality without regularity assumptions on $K$ and $L$. The relation
between Lefschetz operators, Radon transforms, and representations of
convex valuations was investigated in \textbf{\cite{BraunerHofstaetterOrtega2026Lefschetz, SchusterWannerer2015}}.

Our proof follows Lutwak's approach in combining an interchange identity
for mixed quermassintegrals with the Aleksandrov--Fenchel inequality.
For mean section bodies, both ingredients require new forms. The
interchange identity is obtained from a mixed volume representation of
$\overline{\mathrm{M}}_j$, while the required Aleksandrov--Fenchel
inequality concerns averages of mixed volumes with varying reference
bodies. To state it, recall that for $K_1,\ldots,K_m\in\mathcal{K}^n$
and $\lambda_1,\ldots,\lambda_m\geq0$, Minkowski's polynomial formula
gives
\begin{equation} \label{mixed}
V_n(\lambda_1K_1+\cdots+\lambda_mK_m)
=\sum_{j_1,\ldots,j_n=1}^m
V(K_{j_1},\ldots,K_{j_n})\lambda_{j_1}\cdots\lambda_{j_n}.
\end{equation}
The symmetric coefficients $V(K_{j_1},\ldots,K_{j_n})$ are called
\emph{mixed volumes}. As usual, $K[r]$ denotes $r$ copies of $K$,
with an entry omitted when $r=0$, and tuples abbreviate lists of
reference bodies in mixed volumes.

The classical Aleksandrov--Fenchel inequality (see, e.g.,
\textbf{\cite{Aleksandrov1937, schneider93, ShenfeldVanHandel2019}}) states that, for
$K,L\in\mathcal{K}^n$ and any $(n-2)$-tuple $\mathbf{A}$ of convex
bodies,
\[
V(K,L,\mathbf{A})^2
\geq V(K,K,\mathbf{A})V(L,L,\mathbf{A}).
\]

In general, this inequality is not preserved under positive linear
combinations of mixed volumes, and relatively little
appears to be known with regard to this, apart from a few very special results and a general structural
theorem; see, e.g. \textbf{\cite{AbardiaWannerer2015,RossSuessWannerer2025}}. Our second main result shows that the Aleksandrov--Fenchel inequality
is preserved when averaging a particular family of mixed volumes whose
reference bodies include Euclidean balls of lower dimension. For
$u\in S^{n-1}$, let $D_u=B\cap u^\perp$, where $B$ is the
Euclidean unit ball in $\mathbb{R}^n$.

\begin{satz} \label{thm:averagedAF}
Let $2 \leq j \leq n - 1$ and let $\mu$ be a finite non-negative
Borel measure on $S^{n-1}$. For
$K_1,\ldots,K_{j-2} \in \mathcal{K}^n$, put
$\mathbf{C} = (K_1,\ldots,K_{j-2})$, interpreted as the empty tuple when $j=2$, and define
\begin{equation} \label{def:averagedmixedvolume}
\varphi(K,L) = \int_{S^{n-1}}
V(K,L,\mathbf{C},D_u[n-j])\,d\mu(u),
\qquad K,L \in \mathcal{K}^n.
\end{equation}
Then, for all $K,L \in \mathcal{K}^n$,
\begin{equation} \label{ineq:averagedAF}
\varphi(K,L)^2 \geq \varphi(K,K)\varphi(L,L).
\end{equation}
\end{satz}

The proof of Theorem \ref{thm:averagedAF} rests on the observation that
the Aleksandrov--Fenchel inequality for a symmetric bilinear form is
equivalent to its being negative semidefinite on a hyperplane. We first
extend the mixed volumes in (\ref{def:averagedmixedvolume}) to bilinear
forms on $C^2(S^{n-1})$. Using a polarized identity for zonal functions
from \textbf{\cite{BraunerHofstaetterOrtega2024}}, we then show that the
forms corresponding to the different disks $D_u$ are negative
semidefinite on the same hyperplane, determined by a mixed volume with
reference balls. This property is preserved under integration against
a non-negative measure, and hence yields the desired inequality for
the averaged form.

We want to emphasize that this novel method is not built upon pointwise curvature estimates, in contrast to many of the classical arguments.
In recent work, Brauner and Ortega-Moreno~\textbf{\cite{BraunerOrtegaMoreno2026}} address the question of Theorem~\ref{thm:averagedAF} through a more standard approach that is based on a new quadratic mixed discriminant inequality. Since the desired inequality turns out to be false when $j=2$, this case is missing in~\textbf{\cite{BraunerOrtegaMoreno2026}}. The method we introduce in here is fundamentally different in that it does not hinge on such mixed discriminant inequalities. In this way, we not only avoid technical questions of regularity, but are also able to cover the case when $j=2$.

\pagebreak

\centerline{\large{\bf{ \setcounter{abschnitt}{2}
\arabic{abschnitt}. Background material}}}

\reseteqn \alpheqn \setcounter{theorem}{0}

\vspace{0.6cm}

In this section we recall the required facts about mixed volumes and mixed
area measures, including their extension to differences of support functions.
We also specify our normalization of the Radon transforms and introduce the
integral transform for zonal functions used in Section 3. For the convex
geometric background we refer to Schneider \textbf{\cite{schneider93}}.

\vspace{0.2cm}

For $K,L\in\mathcal{K}^n$ and $s,t\geq0$, the
\emph{Minkowski combination} is defined by
$sK+tL=\{sx+ty:x\in K,\ y\in L\}$. Its support function satisfies
\[
h(sK+tL,\cdot)=s\,h(K,\cdot)+t\,h(L,\cdot).
\]
Moreover, for $\vartheta\in\mathrm{SO}(n)$ and $x\in\mathbb{R}^n$,
\[
h(\vartheta K,u)=h(K,\vartheta^{-1}u)
\qquad\mbox{and}\qquad h(K+x,u)=h(K,u)+u\cdot x.
\]
The Hausdorff distance of two convex bodies equals the maximum norm of
the difference of their support functions on $S^{n-1}$. Thus Hausdorff
convergence is equivalent to uniform convergence of support functions.
Two convex bodies are \emph{homothetic} if $L=aK+x$ for some
$a>0$ and $x\in\mathbb{R}^n$, and are \emph{dilates} if $x=0$. Throughout, $B$ denotes the Euclidean unit ball in
$\mathbb{R}^n$ and $\kappa_r$ the volume of the unit ball in
$\mathbb{R}^r$, where $\kappa_0=1$. Spherical Lebesgue measure is
denoted by $du$ and has total mass $n\kappa_n$.

For $f\in C^2(S^{n-1})$, we write
\[
D^2f(u)=\nabla_S^2f(u)+f(u)I_{u^\perp},
\]
where $\nabla_S^2$ is the covariant Hessian on $S^{n-1}$ and
$I_{u^\perp}$ is the identity on the tangent space at $u$. A convex
body $K$ is of class \emph{$C^2_+$} if $h(K,\cdot)\in C^2(S^{n-1})$
and $D^2h(K,\cdot)$ is positive definite at every point of the sphere.
Consequently, for every $f\in C^2(S^{n-1})$, the function $1+tf$
is the support function of a body of class $C^2_+$ whenever $|t|$
is sufficiently small. Choosing a sufficiently large constant $c>0$
also shows that $f=(f+c)-c$ is a difference of support functions.
We will use the fact that bodies of class $C^2_+$ are dense in
$\mathcal{K}^n$.

\vspace{0.2cm}

The mixed volumes recalled in the introduction are symmetric, continuous,
and translation invariant in their arguments, as well as Minkowski
multilinear and monotone with respect to set inclusion in each argument.
We use the normalization $V(K[n])=V_n(K)$ and the notation for
repeated arguments and tuples introduced in Section 1.
For $K,L\in\mathcal{K}^n$ and $0\leq m\leq n-1$, the
\emph{quermassintegrals} and their mixed counterparts are given by
\[
W_m(K)=V(K[n-m],B[m]),\qquad
W_m(K,L)=V(K[n-m-1],B[m],L).
\]
We also set $W_n(K)=\kappa_n$. Thus $W_m(K,K)=W_m(K)$ for
$0\leq m\leq n-1$, and the \emph{intrinsic volumes} $V_i(K)$ are determined by
\begin{equation} \label{viwi}
\kappa_{n-i}V_i(K)=\binom ni W_{n-i}(K),\qquad 0\leq i\leq n.
\end{equation}
In particular, $W_{n-1}(K)=\frac1n\int_{S^{n-1}}h(K,u)\,du$.

For $K,L\in\mathcal{K}^n$ and any $(n-2)$-tuple
$\mathbf{A}=(A_1,\ldots,A_{n-2})$ of convex bodies, the
\emph{Aleksandrov--Fenchel inequality} states that
\begin{equation} \label{eq:classicalAF}
V(K,L,\mathbf{A})^2\geq
V(K,K,\mathbf{A})V(L,L,\mathbf{A}).
\end{equation}
Here, none of the bodies is required to have non-empty interior.
For modern approaches to this inequality through the Bochner method
and Hodge--Riemann relations, we refer to
\textbf{\cite{ShenfeldVanHandel2019, KotrbatyWannerer2022, BernigKotrbatyWannerer}}.
Recent advances on its equality conditions, including Minkowski's
quadratic inequality and the case of convex polytopes, are due to
Shenfeld and van Handel
\textbf{\cite{ShenfeldVanHandel2022, ShenfeldVanHandel2023}}. For further developments concerning the support of mixed area measures
and their relation to equality in the Aleksandrov--Fenchel inequality,
see Hug and Reichert \textbf{\cite{HugReichert2024}}.
A consequence of (\ref{eq:classicalAF}) is the generalized Minkowski
inequality: If $K,L\in\mathcal{K}^n$ have non-empty interiors and
$2\leq i\leq n$, then
\begin{equation} \label{genmink}
W_{n-i}(K,L)^i\geq W_{n-i}(K)^{i-1}W_{n-i}(L),
\end{equation}
with equality if and only if $K$ and $L$ are homothetic. We will
also use the multiplicative form of the Brunn--Minkowski inequality
for quermassintegrals,
\begin{equation} \label{quermassbm}
W_{n-i}(\lambda K+(1-\lambda)L)
\geq W_{n-i}(K)^\lambda W_{n-i}(L)^{1-\lambda},
\qquad 0<\lambda<1,
\end{equation}
for $2\leq i\leq n$ and $K,L$ with non-empty interiors.
Equality holds if and only if $K$ and $L$ are translates.

Associated with $K_1,\ldots,K_{n-1}\in\mathcal{K}^n$ is a finite
non-negative Borel measure $S(K_1,\ldots,K_{n-1},\cdot)$ on
$S^{n-1}$, called the \emph{mixed area measure}, characterized by
\[
V(L,K_1,\ldots,K_{n-1})
=\frac1n\int_{S^{n-1}}h(L,u)\,
 dS(K_1,\ldots,K_{n-1},u),\qquad L\in\mathcal{K}^n.
\]
Mixed area measures are symmetric and Minkowski multilinear in their
body arguments. For $K\in\mathcal{K}^n$ and $0\leq i\leq n-1$,
the \emph{area measure of order $i$} is defined by
$S_i(K,\cdot)=S(K[i],B[n-i-1],\cdot)$. Hence,
\begin{equation} \label{defsi}
W_{n-i}(K,L)=\frac1n\int_{S^{n-1}}h(L,u)\,dS_{i-1}(K,u),
\qquad 1\leq i\leq n.
\end{equation}
By translation invariance of mixed volumes, every mixed area measure
has centroid at the origin, that is,
\[
\int_{S^{n-1}}u\,dS(K_1,\ldots,K_{n-1},u)=0.
\]
Its total mass is $nV(B,K_1,\ldots,K_{n-1})$.

Since every function in $C^2(S^{n-1})$ is a difference of support
functions, mixed volumes extend by multilinearity to arbitrary
$C^2(S^{n-1})$ functions in each of their arguments. This extension
is independent of the chosen representations as differences of support
functions and remains symmetric in all arguments.

More generally, if $K_1,\ldots,K_{n-1}\in\mathcal{K}^n$ and $f$ is
integrable with respect to $S(K_1,\ldots,K_{n-1},\cdot)$, we put
\begin{equation} \label{eq:functionmixedvolume}
V(f,K_1,\ldots,K_{n-1})
=\frac1n\int_{S^{n-1}}f(u)\,
 dS(K_1,\ldots,K_{n-1},u).
\end{equation}
For support functions, this agrees with the usual mixed volume and,
for $f\in C^2(S^{n-1})$, with the multilinear extension above.
In particular, (\ref{eq:functionmixedvolume}) is defined for every
continuous function on $S^{n-1}$. Related questions concerning
Brunn--Minkowski inequalities, monotonicity, and concavity for
functionals defined by integration against area measures were studied
by Colesanti, Hug, and Saor\'in-G\'omez
\textbf{\cite{ColesantiHugSaorinGomez2014, ColesantiHugSaorinGomez2017}}.

\vspace{0.2cm}

We next recall that the invariant measures on the Grassmannians and
on the fibers defining the Radon transforms are normalized to have
total mass one. With this convention, the Kubota formula is
\begin{equation} \label{eq:kubotabackground}
\int_{\mathrm{Gr}_{j,n}}\mathrm{vol}_j(K|E)\,dE
=\frac{\kappa_j}{\kappa_n}W_{n-j}(K),\qquad 1\leq j\leq n-1.
\end{equation}
For $u\in S^{n-1}$, let $D_u=B\cap u^\perp$ denote the
$(n-1)$-dimensional unit disk orthogonal to $u$. Although $D_u$
has empty interior in $\mathbb{R}^n$, it may be used as an argument
of mixed volumes and mixed area measures. Note that $D_u=D_{-u}$
and that $u\mapsto D_u$ is continuous in the Hausdorff metric.

Recall that a map $\Phi:\mathcal{K}^n\to\mathcal{K}^n$ is a
\emph{Minkowski valuation} if
\[
\Phi(K\cup L)+\Phi(K\cap L)=\Phi K+\Phi L
\]
whenever $K,L,K\cup L\in\mathcal{K}^n$. It is called \emph{even}
if $\Phi(-K)=\Phi K$. The operators $\overline{\mathrm{M}}_j$
are continuous even Minkowski valuations. Indeed, for each
$E\in\mathrm{Gr}_{j,n}$, the map $K\mapsto\mathrm{vol}_j(K|E)$
is a continuous translation invariant valuation, and the assertion
follows by applying $R_{j,1}$ in (\ref{defMimaps}).
It also follows directly from (\ref{defMimaps}) that
$\overline{\mathrm{M}}_j$ is translation invariant, commutes with
rotations, is monotone with respect to set inclusion, and is homogeneous
of degree $j$. Moreover, $\overline{\mathrm{M}}_jK$ is origin-symmetric
for every $K\in\mathcal{K}^n$, and our normalization gives
$\overline{\mathrm{M}}_jB=\kappa_jB$.

To relate these maps to the original mean section operators, let
$\mathrm{AGr}_{k,n}$ denote the space of affine $k$-dimensional
subspaces of $\mathbb{R}^n$, equipped with its rigid motion invariant
measure. This measure is obtained by integrating Lebesgue measure on
$E^\perp$ against invariant probability measure on $\mathrm{Gr}_{k,n}$.
For $2\leq k\leq n$, the normalized mean section operator is given by
\[
h(\mathrm{M}_kK,u)
=\int_{\mathrm{AGr}_{k,n}}
h((K\cap F)-s(K\cap F),u)\,dF,
\]
where the integrand is zero for empty sections and
\[
s(A)=\frac1{\kappa_n}\int_{S^{n-1}}h(A,v)v\,dv
\]
is the \emph{Steiner point} of $A\in\mathcal{K}^n$.
The operator $\mathrm{M}_k$ has degree $n+1-k$. Up to a positive
constant depending only on $n$ and $j$, the operator
$\overline{\mathrm{M}}_j$ is the even part
$\frac12(\mathrm{M}_{n+1-j}K+\mathrm{M}_{n+1-j}(-K))$; see
\textbf{\cite{GoodeyWeil1992, Goodey1998}}.
Relations between section and projection means were studied by Goodey,
Kiderlen, and Weil \textbf{\cite{GoodeyKiderlenWeil1998}}.
For convolution representations and the connection with the spherical
sine transform, we refer to
\textbf{\cite{SchusterWannerer2018, MareschSchuster2012}}, and for
recent results on fixed points to \textbf{\cite{BraunerOrtegaMoreno2025}}.

In the final part of this section, we introduce the transform of zonal
functions which relates mixed volumes with reference balls to those
with reference disks. Fix $\bar{e}\in S^{n-1}$ and identify
$\mathrm{SO}(n-1)$ with the subgroup of $\mathrm{SO}(n)$ fixing
$\bar{e}$. A function on $S^{n-1}$ is called \emph{zonal} if it is
invariant under this subgroup. Every continuous zonal function $f$
can be identified with a unique function $\bar{f}\in C[-1,1]$ by
$f(u)=\bar{f}(u\cdot\bar{e})$. We write $D=D_{\bar{e}}$ and, for
$1\leq i\leq n-1$, set
\[
S_i(K,D,\cdot)=S(K[i],D[n-i-1],\cdot).
\]
Mixed area measures with disk reference bodies are studied in
\textbf{\cite{BraunerHofstaetterOrtegaDisk2025}}, while mixed
Christoffel--Minkowski problems and transformations of mixed area
measures for bodies of revolution are treated in
\textbf{\cite{BraunerHofstaetterOrtegaRevolution2025}}.

Following \textbf{\cite{BraunerHofstaetterOrtega2024}}, for $\alpha>0$
we define $T_\alpha:C[-1,1]\to C[-1,1]$ by
\begin{equation} \label{def:kerneltransform}
(T_\alpha\bar{f})(s)
=(1-s^2)^{\alpha/2}\bar{f}(s)
+\alpha s\int_0^s\bar{f}(t)(1-t^2)^{(\alpha-2)/2}\,dt,
\qquad -1<s<1.
\end{equation}
At $s=\pm1$, the right-hand side is understood by continuous
extension, which exists since $\alpha>0$ and $\bar{f}$ is continuous.
We also put $T_0\bar{f}=\bar{f}$. The relation between $T_{n-i-1}$
and mixed volumes, in the polarized form needed here, is recalled
in Lemma \ref{lem:zonalrepresentations}.

\vspace{0.6cm}

\centerline{\large{\bf{ \setcounter{abschnitt}{3}
\arabic{abschnitt}. Auxiliary Results}}}

\reseteqn \alpheqn \setcounter{theorem}{0}

\vspace{0.6cm}

We first recall an elementary fact about symmetric bilinear forms
which will be used in the proof of Theorem \ref{thm:averagedAF}.

\begin{lem} \label{lem:negativehyperplane}
Let $X$ be a non-zero real vector space and let $[\cdot,\cdot]$
be a symmetric bilinear form on $X$. Then the following statements
are equivalent:
\begin{itemize}
\item[\upshape(a)] For all $x,y \in X$ with $[y,y] \geq 0$,
\[
[x,y]^2 \geq [x,x][y,y].
\]
\item[\upshape(b)] There exists $z \in X$ such that $[x,x] \leq 0$
for every $x \in X$ satisfying $[x,z] = 0$.
\item[\upshape(c)] There exists a hyperplane $H \subseteq X$ on which
$[\cdot,\cdot]$ is negative semidefinite.
\end{itemize}
\end{lem}

\noindent {\it Proof.} To see that (b) implies (c), note that if the linear
functional $x \mapsto [x,z]$ is non-zero, its kernel is a hyperplane
with the required property. Otherwise, the form is negative semidefinite
on $X$, and we may choose any hyperplane.

Suppose now that (c) holds and let $x,y \in X$ with $[y,y] \geq 0$.
If $[y,y] = 0$, the inequality in (a) is immediate. If $[y,y] > 0$
and the inequality fails, then
\[
[x,x][y,y] - [x,y]^2 > 0.
\]
Hence, by Sylvester's criterion, the matrix
\[
\begin{pmatrix}
[x,x] & [x,y] \\
[x,y] & [y,y]
\end{pmatrix}
\]
is positive definite. Since $H$ has codimension one, there exist
$a,b \in \mathbb{R}$, not both zero, such that $ax+by \in H$.
Positive definiteness also implies that $x$ and $y$ are linearly
independent, so $ax+by\neq0$. Thus $[ax+by,ax+by]>0$, contradicting
the choice of $H$. Thus (c) implies (a).

Finally, to show that (a) implies (b), first note that if the form is negative semidefinite
on $X$, then (b) holds with $z = 0$. Otherwise, choose $z \in X$
with $[z,z] > 0$. For every $x \in X$ with $[x,z] = 0$, (a) gives
\[
0 = [x,z]^2 \geq [x,x][z,z],
\]
and hence $[x,x] \leq 0$, as required.
\hfill $\blacksquare$

\vspace{0.3cm}

We next state a polarized version of
\textbf{\cite[\textnormal{Corollary 3.11}]{BraunerHofstaetterOrtega2024}}
which relates mixed volumes with reference balls to those with reference
disks. Here, $T_\alpha$ is the transform defined in
(\ref{def:kerneltransform}) and $D=B\cap\bar{e}^{\perp}$.

\begin{lem} \label{lem:zonalrepresentations}
Let $1\leq i\leq n-1$ and let $f,g\in C(S^{n-1})$ be zonal,
with $f(u)=\bar{f}(u\cdot\bar{e})$ and
$g(u)=\bar{g}(u\cdot\bar{e})$. If
$\bar{g}=T_{n-i-1}\bar{f}$, then, for every tuple
$\mathbf{C}=(C_1,\ldots,C_i)$ of convex bodies in $\mathcal{K}^n$,
\begin{equation} \label{eq:polarizedkerneltransform}
V(f,\mathbf{C},B[n-i-1])
=V(g,\mathbf{C},D[n-i-1]).
\end{equation}
\end{lem}

\noindent {\it Proof.} By
\textbf{\cite[\textnormal{Corollary 3.11}]{BraunerHofstaetterOrtega2024}},
for every $K\in\mathcal{K}^n$,
\[
\int_{S^{n-1}}f(u)\,dS(K[i],B[n-i-1],u)
=\int_{S^{n-1}}g(u)\,dS(K[i],D[n-i-1],u).
\]
Equivalently,
\[
V(f,K[i],B[n-i-1])=V(g,K[i],D[n-i-1]).
\]
Replacing $K$ by $t_1C_1+\cdots+t_iC_i$, with $t_1,\ldots,t_i\geq0$,
and comparing coefficients of $t_1\cdots t_i$, we obtain
(\ref{eq:polarizedkerneltransform}). \hfill $\blacksquare$

\vspace{0.3cm}

Note that, by multilinearity, (\ref{eq:polarizedkerneltransform}) also holds when any of the bodies in
$\mathbf{C}$ is replaced by a difference of support functions. In
particular, we may use arbitrary $C^2(S^{n-1})$ functions in these
arguments. Moreover, by rotation invariance of mixed volumes, the identity holds
with any pole $u\in S^{n-1}$, provided that $D$ is replaced by $D_u$
and the zonal functions are rotated accordingly.

\vspace{0.3cm}

The following mixed volume representation of mean section bodies will
be used to obtain an analogue of Lutwak's interchange formula. The
constant corresponds to the normalization in (\ref{defMimaps}).

\begin{koro} \label{cor:meansectionmixedvolume}
Let $1 \leq j \leq n-1$. For every $K \in \mathcal{K}^n$,
\begin{equation} \label{eq:meansectionmixedvolume}
h(\overline{\mathrm{M}}_jK,u)
= \frac{n\kappa_{j-1}}{j\kappa_{n-1}}
V(K[j],D_u[n-j]), \qquad u \in S^{n-1},
\end{equation}
where $D_u = B \cap u^\perp$ and $\kappa_0 = 1$.
\end{koro}

\noindent {\it Proof.} For $u \in S^{n-1}$, let
$\mathrm{Gr}_{j,n}^{u} = \{E \in \mathrm{Gr}_{j,n}: u \in E\}$,
endowed with its invariant probability measure $\nu_j^u$.
By the definition of $\overline{\mathrm{M}}_j$,
\[
h(\overline{\mathrm{M}}_jK,u)
= \int_{\mathrm{Gr}_{j,n}^{u}}\mathrm{vol}_j(K|E)\,d\nu_j^u(E).
\]
Applying the Kubota type formula from
\textbf{\cite[\textnormal{Theorem 5.1}]{BraunerHofstaetterOrtega2024}}
with pole $u$ now yields (\ref{eq:meansectionmixedvolume}). \hfill $\blacksquare$

\vspace{0.3cm}

Recall that Lutwak's interchange identity for projection bodies
\textbf{\cite{lutwak93}} states that, for $K,L\in\mathcal{K}^n$,
$2\leq i\leq n$, and $1\leq j\leq n-1$,
\[
W_{n-i}(K,\Pi_jL)=W_{n-j-1}(L,\Pi_{i-1}K),
\]
where the operators $\Pi_r$ are normalized such that
$\Pi_rB=\kappa_{n-1}B$. Equivalently, with the normalization in
(\ref{defPimaps}),
\[
W_{n-i}(K,\overline{\Pi}_jL)
=\frac{\kappa_j}{\kappa_{i-1}}
W_{n-j-1}(L,\overline{\Pi}_{i-1}K).
\]
More general symmetry properties of rigid motion invariant homogeneous
bivaluations were established in \textbf{\cite{AleskerBernigSchuster2011}} and applied to interchange
identities for Minkowski valuations in
\textbf{\cite{ParapatitsSchuster2012}}.
The analogue for mean section bodies takes a different form.
Combining Corollary \ref{cor:meansectionmixedvolume} with the integral
representation of mixed quermassintegrals yields the following identity.

\begin{lem} \label{lem:lutwakmeansection}
Let $1 \leq i \leq n$ and $2 \leq j \leq n-1$. For all
$K,L \in \mathcal{K}^n$,
\begin{equation} \label{eq:lutwakmeansection}
W_{n-i}(K,\overline{\mathrm{M}}_jL)
= \frac{\kappa_{j-1}}{j\kappa_{n-1}}
\int_{S^{n-1}} V(L[j],D_u[n-j])\,dS_{i-1}(K,u),
\end{equation}
where $D_u = B \cap u^\perp$.
\end{lem}

\noindent {\it Proof.} By (\ref{defsi}),
\[
W_{n-i}(K,\overline{\mathrm{M}}_jL)
=\frac1n\int_{S^{n-1}}
h(\overline{\mathrm{M}}_jL,u)\,dS_{i-1}(K,u).
\]
Substituting (\ref{eq:meansectionmixedvolume}) proves the assertion.
\hfill $\blacksquare$

\vspace{1cm}

\centerline{\large{\bf{ \setcounter{abschnitt}{4}
\arabic{abschnitt}. Proof of Theorem 2}}}

\reseteqn \alpheqn \setcounter{theorem}{0}

\vspace{0.6cm}

We now prove Theorem \ref{thm:averagedAF} by identifying a common
hyperplane on which the bilinear forms associated with the disks $D_u$
are negative semidefinite.

\vspace{0.5cm}

\noindent {\it Proof of Theorem \ref{thm:averagedAF}.} Put $m=n-j$, so that
$1\leq m\leq n-2$. We first assume that the bodies in
$\mathbf{C}=(K_1,\ldots,K_{j-2})$ have non-empty interiors.
For $j=2$, there are no reference bodies. Using the extension of mixed
volumes described in Section 2, for each $u\in S^{n-1}$ we define
a symmetric bilinear form on $X=C^2(S^{n-1})$ by
\begin{equation} \label{eq:diskbilinearform}
Q_u(f,g)=V(f,g,\mathbf{C},D_u[m]).
\end{equation}
We will show that every $Q_u$ is negative semidefinite on
\begin{equation} \label{eq:commonhyperplane}
H=\{f\in X: V(f,\mathbf{C},B[m+1])=0\}.
\end{equation}
The defining functional is non-zero, since
$V(B,\mathbf{C},B[m+1])>0$, and hence $H$ is a hyperplane.

First, we record the consequence of the classical
Aleksandrov--Fenchel inequality needed for the forms $Q_u$.
Since the bodies in $\mathbf{C}$ have non-empty interiors and
$D_u$ is an $(n-1)$-dimensional disk,
\[
Q_u(1,1)=V(B,B,\mathbf{C},D_u[m])>0.
\]
Indeed, after translating the bodies in $\mathbf{C}$, they contain
positive multiples of $B$, and the assertion follows from the
monotonicity of mixed volumes and $V(B[j],D_u[m])>0$.
For any $f\in X$, the function $1+tf$ is the support function of
a convex body for all sufficiently small $|t|$. Applying the
Aleksandrov--Fenchel inequality to this body and $B$ yields
\[
Q_u(1+tf,1)^2\geq Q_u(1+tf,1+tf)Q_u(1,1).
\]
Expanding both sides and cancelling gives
\[
Q_u(f,1)^2\geq Q_u(f,f)Q_u(1,1).
\]
Consequently, $Q_u$ is negative semidefinite on the hyperplane
$\{f\in X:Q_u(f,1)=0\}$. By Lemma
\ref{lem:negativehyperplane}, it follows that
\begin{equation} \label{eq:extendeddiskAF}
Q_u(f,g)^2\geq Q_u(f,f)Q_u(g,g)
\end{equation}
for all $f,g\in X$ with $Q_u(g,g)\geq0$.

To identify the common hyperplane in (\ref{eq:commonhyperplane}),
we construct a function $p_u\in X$ such that
\begin{equation} \label{eq:testfunctionidentity}
Q_u(p_u,f)=V(f,\mathbf{C},B[m+1]),\qquad f\in X,
\end{equation}
and $Q_u(p_u,p_u)>0$. Define $\bar{p}$ on $(-1,1)$ by
\begin{equation} \label{eq:testfunctionprofile}
\bar{p}(s)=(T_m1)(s)
=(1-s^2)^{m/2}
+ms\int_0^s(1-t^2)^{(m-2)/2}\,dt,
\end{equation}
and put $p_u(v)=\bar{p}(u\cdot v)$. At the endpoints, we use the
continuous extension
\[
\bar{p}(1)=\bar{p}(-1)
=m\int_0^1(1-t^2)^{(m-2)/2}\,dt>0.
\]
The function $\bar{p}$ is even and strictly positive on $[-1,1]$.
Moreover, on $(-1,1)$,
\[
\bar{p}'(s)=m\int_0^s(1-t^2)^{(m-2)/2}\,dt,
\qquad
\bar{p}''(s)=m(1-s^2)^{(m-2)/2}.
\]
These identities also show that $p_u\in C^2(S^{n-1})$.
For completeness, near the pole $u$, with $s=\cos\theta$, write
\[
\bar{p}(\cos\theta)
=\bar{p}(1)\cos\theta+r(\theta).
\]
The identities above imply
$r^{(\ell)}(\theta)=O(\theta^{m+2-\ell})$ for $0\leq\ell\leq2$.
Thus the remainder, as a radial function in local coordinates at the
pole, has two continuous derivatives there. The argument at $-u$
is identical by evenness.

We now apply Lemma \ref{lem:zonalrepresentations} with $i=j-1$
and pole $u$. Since $n-i-1=m$ and $\bar{p}=T_m1$, its polarized
identity gives
\[
V(1,f,\mathbf{C},B[m])
=V(p_u,f,\mathbf{C},D_u[m]),\qquad f\in X.
\]
Here we used the extension to differences of support functions
following the lemma. Since $1=h(B,\cdot)$, this is precisely
(\ref{eq:testfunctionidentity}).
In particular,
\begin{eqnarray*}
Q_u(p_u,p_u)
&=&V(p_u,\mathbf{C},B[m+1])\\
&=&\frac1n\int_{S^{n-1}}p_u(v)\,
 dS(\mathbf{C},B[m+1],v)>0.
\end{eqnarray*}
Here we used (\ref{eq:functionmixedvolume}) and the fact that the
mixed area measure is non-negative with positive total mass, while
$p_u$ is strictly positive.

For every $f\in H$, (\ref{eq:testfunctionidentity}) now gives
$Q_u(p_u,f)=0$. Applying (\ref{eq:extendeddiskAF}) with $g=p_u$,
we obtain
\[
0=Q_u(f,p_u)^2\geq Q_u(f,f)Q_u(p_u,p_u),
\]
so that $Q_u(f,f)\leq0$. This proves the asserted negative
semidefiniteness on $H$, for every $u\in S^{n-1}$.

Define a symmetric bilinear form on $X$ by
\[
Q(f,g)=\int_{S^{n-1}}Q_u(f,g)\,d\mu(u).
\]
This integral is well defined: expressing $f$ and $g$ as differences
of support functions reduces the integrand to a finite linear
combination of continuous mixed volumes, and $\mu$ is finite.
Since $\mu$ is non-negative, for every $f\in H$,
\[
Q(f,f)=\int_{S^{n-1}}Q_u(f,f)\,d\mu(u)\leq0.
\]
A second application of Lemma \ref{lem:negativehyperplane}
therefore gives
\[
Q(f,g)^2\geq Q(f,f)Q(g,g)
\]
whenever $Q(g,g)\geq0$. In particular, if $K,L$ are of class
$C^2_+$, their support functions belong to $X$, and the
non-negativity of mixed volumes implies $Q(h(L,\cdot),h(L,\cdot))\geq0$.
Since $Q(h(K,\cdot),h(L,\cdot))=\varphi(K,L)$, we obtain
\[
\varphi(K,L)^2\geq\varphi(K,K)\varphi(L,L).
\]

Finally, approximate $K$ and $L$ in the Hausdorff metric by bodies
of class $C^2_+$ and replace each reference body $K_r$ by
$K_r+\varepsilon B$, $\varepsilon>0$. These reference bodies have
non-empty interiors and converge to $K_r$ as $\varepsilon\to0$.
By continuity of mixed volumes and compactness of the family of disks
$D_u$, all three integrands converge uniformly in $u\in S^{n-1}$.
The corresponding integrals therefore converge against the finite
measure $\mu$. Passing to the limit proves the desired inequality
for arbitrary convex bodies. \hfill $\blacksquare$

\vspace{1cm}

\centerline{\large{\bf{ \setcounter{abschnitt}{5}
\arabic{abschnitt}. Proof of Theorem 1}}}

\reseteqn \alpheqn \setcounter{theorem}{0}

\vspace{0.6cm}

For the proof of Theorem \ref{thm1}, we use the abbreviations
$\mathrm{M}=\overline{\mathrm{M}}_j$ and
$K_\lambda=\lambda K+(1-\lambda)L$.

\vspace{0.5cm}

\noindent {\it Proof of Theorem \ref{thm1}.} By the properties recalled
in Section 2, $\mathrm{M}A$ is origin-symmetric and has non-empty
interior whenever $A$ does. Indeed, after a translation,
$A\supseteq rB$ for some $r>0$, and monotonicity gives
$\mathrm{M}A\supseteq r^j\kappa_jB$.
Averaging (\ref{defMimaps}) over $S^{n-1}$ and applying the
classical Kubota formula gives
\begin{equation} \label{eq:meansectionmeanwidth}
W_{n-1}(\mathrm{M}A)=\kappa_j W_{n-j}(A),
\qquad A\in\mathcal{K}^n.
\end{equation}
Indeed, averaging the fiber measures in $R_{j,1}$ gives invariant
probability measure on $\mathrm{Gr}_{j,n}$, so the assertion follows
from (\ref{eq:kubotabackground}) and
$W_{n-1}(\mathrm{M}A)=\frac1n\int_{S^{n-1}}h(\mathrm{M}A,u)\,du$. For $i=1$, the desired inequality and its equality
conditions are therefore immediate from
(\ref{eq:meansectionmeanwidth}) and (\ref{quermassbm}).

Suppose now that $2\leq i\leq n$ and fix a convex body $Q$ with
non-empty interior. For $A_1,\ldots,A_j\in\mathcal{K}^n$, define
\[
F_Q(A_1,\ldots,A_j)
=\frac{\kappa_{j-1}}{j\kappa_{n-1}}
\int_{S^{n-1}}V(A_1,\ldots,A_j,D_u[n-j])\,dS_{i-1}(Q,u).
\]
This functional is symmetric, multilinear and non-negative.
By Lemma \ref{lem:lutwakmeansection},
\begin{equation} \label{eq:interchangeFQ}
F_Q(A[j])=W_{n-i}(Q,\mathrm{M}A).
\end{equation}
The measure $S_{i-1}(Q,\cdot)$ is finite and non-negative. Hence Theorem \ref{thm:averagedAF} gives
\begin{equation} \label{eq:FQAF}
\begin{split}
F_Q(A_1,A_2,A_3,\ldots,A_j)^2
\geq {}&F_Q(A_1,A_1,A_3,\ldots,A_j)\\
&{}\cdot F_Q(A_2,A_2,A_3,\ldots,A_j).
\end{split}
\end{equation}
For $j=2$, the list of remaining arguments is empty.

For bodies $K,L$ with non-empty interiors, put
$a_r=F_Q(K[j-r],L[r])$, $0\leq r\leq j$. These numbers are
positive: after translations, $K$ and $L$ contain positive multiples
of $B$, the mixed volumes $V(B[j],D_u[n-j])$ are positive and
independent of $u$, and $S_{i-1}(Q,\cdot)$ has positive total mass.
Applying (\ref{eq:FQAF}) with the remaining arguments chosen
from $K$ and $L$ yields $a_r^2\geq a_{r-1}a_{r+1}$ for
$1\leq r\leq j-1$. Consequently,
\[
a_r\geq a_0^{(j-r)/j}a_j^{r/j},\qquad 0\leq r\leq j.
\]
By multilinearity and the binomial theorem,
\begin{eqnarray*}
F_Q(K_\lambda[j])
&=&\sum_{r=0}^j\binom jr\lambda^{j-r}(1-\lambda)^r a_r\\
&\geq&\bigl(\lambda a_0^{1/j}+(1-\lambda)a_j^{1/j}\bigr)^j\\
&\geq&a_0^\lambda a_j^{1-\lambda},
\end{eqnarray*}
where the last step is the weighted arithmetic-geometric mean
inequality. Using (\ref{eq:interchangeFQ}), we conclude that
\begin{equation} \label{eq:meansectionmixedlogconcavity}
W_{n-i}(Q,\mathrm{M}K_\lambda)
\geq W_{n-i}(Q,\mathrm{M}K)^\lambda
W_{n-i}(Q,\mathrm{M}L)^{1-\lambda}.
\end{equation}

We now choose $Q=\mathrm{M}K_\lambda$ and write
$w=W_{n-i}(Q)>0$. The general Minkowski inequality
(\ref{genmink}) gives
\begin{eqnarray*}
W_{n-i}(Q,\mathrm{M}K)^i
&\geq& w^{i-1}W_{n-i}(\mathrm{M}K),\\
W_{n-i}(Q,\mathrm{M}L)^i
&\geq& w^{i-1}W_{n-i}(\mathrm{M}L).
\end{eqnarray*}
Combining these inequalities with
(\ref{eq:meansectionmixedlogconcavity}), we obtain
\[
w^i\geq w^{i-1}
W_{n-i}(\mathrm{M}K)^\lambda
W_{n-i}(\mathrm{M}L)^{1-\lambda}.
\]
Dividing by $w^{i-1}$ proves the desired inequality for
$W_{n-i}$ and therefore also for $V_i$.

It remains to establish the equality conditions. Suppose that
equality holds. Then equality must hold in both applications of
the general Minkowski inequality. Since all three mean section
bodies have non-empty interiors, they are homothetic. Their
origin symmetry implies that they are, in fact, dilates. Thus,
there exist $a,b>0$ such that
\[
\mathrm{M}K=aQ,\qquad \mathrm{M}L=bQ.
\]
The equality in the asserted log-concavity inequality, together
with the homogeneity of $W_{n-i}$, gives
$a^\lambda b^{1-\lambda}=1$. Applying
(\ref{eq:meansectionmeanwidth}), we find
\begin{eqnarray*}
W_{n-j}(K)&=&\frac{a}{\kappa_j}W_{n-1}(Q),\\
W_{n-j}(L)&=&\frac{b}{\kappa_j}W_{n-1}(Q),\\
W_{n-j}(K_\lambda)&=&\frac1{\kappa_j}W_{n-1}(Q).
\end{eqnarray*}
Consequently,
\[
W_{n-j}(K_\lambda)
=W_{n-j}(K)^\lambda W_{n-j}(L)^{1-\lambda}.
\]
By the equality conditions of (\ref{quermassbm}), $K$ and $L$
are translates of each other. Conversely, if $K$ and $L$ are
translates, then $K_\lambda$ is a translate of $K$, and the
translation invariance of $\mathrm{M}$ implies equality.
\hfill $\blacksquare$

\vspace{0.5cm}

\noindent {\bf Acknowledgments.} The authors are grateful to Ramon
van Handel for fruitful discussions during the preparation of this article.

The first-named author was supported by the Deutsche Forschungsgesellschaft (DFG), project number~520350299.
The second-named author was partially supported by grant PID2025-169631NB-I00, funded by the Spanish Ministry of Science, Innovation and Universities and the State Research Agency (MICIU/AEI/10.13039/501100011033).
The third-named author was supported by the Austrian science fund (FWF), project number P37030-N.

\end{document}